\documentclass{ws-rmta}
\usepackage{graphicx}
\usepackage{dsfont}

\begin{document}

\title{ON DETERMINING THE NUMBER OF SPIKES IN A HIGH-DIMENSIONAL SPIKED POPULATION MODEL}

\author{DAMIEN PASSEMIER}

\address{IRMAR, Universit\'e de Rennes 1, Campus de Beaulieu\\35042 Rennes Cedex, France\\
\email{damien.passemier@univ-rennes1.fr}}

\author{JIAN-FENG YAO}

\address{Department of Statistics and Actuarial Science\\The University of Hong Kong\\ Pokfulam, Hong Kong\\
\email{jeffyao@hku.hk}}

\maketitle

\begin{history}
31 March 2011\footnote{Preprint of an article submitted for consideration in ``Random Matrices: Theory and Applications (RMTA)" \textcopyright 2010 [copyright World Scientific Publishing Company] \url{http://www.worldscinet.com/rmta/}}
\end{history}

\begin{abstract}
In a spiked population model, the population covariance matrix has all its eigenvalues equal to units except for a few fixed eigenvalues (spikes). Determining the number of spikes is a fundamental problem which appears in many scientific fields, including signal processing (linear mixture model) or economics (factor model).  Several recent papers studied the asymptotic behavior of the eigenvalues of the sample covariance matrix (sample eigenvalues) when the dimension of the observations and the sample size both grow to infinity so that their ratio converges to a positive constant. Using  these results, we propose a new estimator based on the difference between two consecutive sample eigenvalues.
\end{abstract}



\section{Introduction}

In a spiked population model, the population covariance matrix has all its eigenva\-lues equal to units except for a few fixed eigenvalues (spikes). This model appears in many scientific fields often with different names. In economics, it is called ``factors model'' within the Ross Arbitrage Pricing Theory (APT) and the aim is to relate observed data (assets) to a small dimensional set of unobserved variables which are then estimated \cite{Ross}. In physics of mixture, ``linear mixture model" are naturally considered for various phenomena \cite{Naes}. In wireless communication, a signal emitted by a source is modulated and received by an array of antennas which will permit the reconstruction of the original signal.\\

An important question to be addressed under this model is how many factors/ components/signals there are. It is generally a first step preliminary to any further study such as estimation and forecasting.\\

Many methods for determining the number of factors have been developed, based on the minimum description length (MDL), Bayesian model selection or Bayesian Information Criteria (BIC) (See \cite{Bai-Ng}). Nevertheless, these methods are based on asymptotic expansions for large sample size and may not perform well when the dimension of the data $p$ is large compared to the sample size $n$. To avoid this problem of high dimension, several methods have been recently proposed using the random matrix theory, such as Harding \cite{Harding} or Onatski \cite{Onatski} in economics, and Kritchman \& Nadler \cite{Nadler} in array processing or chemometrics literature.\\

In this paper, we present a new estimator for the number of spikes from high-dimensional data. Our approach is based on the results of Bai \& Yao \cite{Bai-Yao} and Paul \cite{Paul} which give the limiting distributions of the extreme eigenvalues of a sample covariance matrix coming from a spiked population model, and a recent result of Benaych-Georges, Guionnet \& Maida \cite{Maida}. The obtained results are presented in Section 3.\\

The remaining sections of the paper are organized as follows. In Section~2, we introduce the spiked population model, and recall known results on the almost sure limits of extreme eigenvalues which lead to the idea of our estimator. In Section~3 we define precisely our estimator and prove its consistency in the case of simple spikes with known variance. Next we give a method of estimation in the case of simple spikes with unknown variance. In Section 4, we define the factor/linear mixture model that we link to the spiked population model and we compare our method to those of Harding \cite{Harding} and Kritchman \& Nadler \cite{Nadler}. We consider the case of spikes with greater multiplicity in Section 5. Finally, we discuss the extension to the generalized spiked population model. Throughout the paper, simulation experiments are conducted to access the quality of the proposed estimation.

\section{Spiked Population Model}

We consider $\mathsf{x}=EV^\frac{1}{2}\mathsf{y}$, where $\mathsf{y} \in \mathbb{R}^p$ is a zero-mean random vector of i.i.d. components, $E$ is an orthogonal matrix and
$$V=\mbox{cov}(\mathsf{x})= \sigma^2 \left ( \begin{array}{cc} \Sigma_{q_0} & 0 \\ 
								      0 & I_{p-q_0} \end{array} \right )\text{,}$$
where $\Sigma$ has $K$ non null and non unit eigenvalues $(\alpha_k)_{1\le k \le K}$ with respective multiplicity $(n_k)_{1\le k \le K}$ ($n_1+\dots+n_K=q_0$). Therefore, the eigenvalues of the population covariance matrix $V$ are unit except the $\alpha_j$, called spike eigenvalues. Notice that, if the observations are Gaussian, we may assume that $V$ is diagonal by using a suitable orthogonal transformation.

Let  $(\mathsf{x}_i)_{(1 \le i \le n)}$ be $n$ independent copies of $\mathsf{x}$. The sample covariance matrix is
$$\mathsf{S}_n= \frac{1}{n} \sum_{i=1}^n \mathsf{x}_i \mathsf{x}_i^*\text{.}$$

It is assumed in the sequel that $q_0$ is fixed, and $p$ and $n$ are related so that when $n \rightarrow +\infty$, $\frac{p}{n} \rightarrow c > 0$. Moreover, we assumed that $\alpha_1 > \dots > \alpha_K > 1+\sqrt{c}$ for all $i \in \{1,\dots,K\}$. For $\alpha \ne 1$, we define the function
$$\phi(\alpha) = \alpha + \frac{c\alpha}{\alpha-1}\text{.}$$

Let $\lambda_{n,1} \ge \lambda_{n,2} \ge \dots \ge \lambda_{n,p}$  be the eigenvalues of the sample covariance matrix $\mathsf{S}_n$. Let $s_i = n_1+\dots+n_i$ for $1\le i \le K$. Baik and Silverstein \cite{Baik-Silverstein} proved that, under a moment condition on $\mathsf{x}$, for each $k \in \{1,\dots,K\}$ and $s_{k-1} < j \le s_k$ almost surely,
$$\lambda_{n,j}\longrightarrow \sigma^2\phi(\alpha_k)\text{.}$$

In other words, with the hypotheses that $\alpha_k > 1+\sqrt{c}$ for all $k$, and has multiplicity $n_k$, then $\phi(\alpha_k)$ is the limit of $n_k$ packed sample eigenvalue $\{\lambda_{n,j}$, $s_{k-1} +1 \le j \le s_k\}$. They also prove that for all $1\le i \le L$ with a prefixed range $L$ almost surely,
$$\lambda_{n,q_0+i} \rightarrow b=\sigma^2(1+\sqrt{c})^2 \text{.}$$

Our aim is to estimate $q_0$ when only $\mathsf{S}_n$ is known. The idea is to use, as suggested in Onatski \cite{Onatski}, differences between consecutives eigenvalues
$$\delta_{n,j}=\lambda_{n,j}-\lambda_{n,j+1}\text{.}$$
Indeed, applying the results quoted above it is easy to see that a.s. if $j \ge q_0$, $\delta_{n,j} \rightarrow 0$ while when $j < q_0$, $\delta_{n,j}$ tends to a positive limit if the $\alpha_k$ are different. Thus it is possible to detect $q_0$ from index-numbers $j$ where $\delta_{n,j}$ becomes small.

\section{Case of Simple Spikes with Known Variance $\sigma^2$}

In this section, we suppose that $\sigma$ is known and that all the spikes are simple, i.e $n_1=\dots=n_K=1$. Under these hypotheses the population eigenvalues are
$$\mbox{spec}(V)=  \sigma^2( \underbrace{\alpha_1,\cdots,\alpha_{q_0}}_{q_0},\underbrace{1,\cdots,1}_{p-q_0}  )\text{.}$$
We also need the following assumption:

\begin{assumption} \label{hypothese2}
The entries $\mathsf{y}^i$ of the random vector $\mathsf{y}$ have a symmetric law and a sub-exponential decay, that is there exists positive constants C, C' such that, for all $t\ge \mbox{C'}$,
$$\mathbb{P}(|\mathsf{y}^i| \ge t^{C}) \le e^{-t}\text{.}$$
\end{assumption}
\noindent Especially, the Gaussian vectors satisfy this hypothesis.\\

As stated previously the main observation is that when one follows the sample eigenvalues in a descending order, the successive spacings $\delta_{n,j}$ shrink to small values when approaching non-spiked values. Therefore, our estimation method will use a carefully determined threshold $d_n$. We propose to estimate $q_0$ by the following
$$\hat{q}_n=\mbox{max}\{ j \in \{1,\dots,s\} : \forall k \in \{1,\dots,j\}\mbox{, }\delta_{n,j} \ge d_n \mbox{ and }\delta_{n,j+1} < d_n \}\text{,}$$
where $s >q_0$ is a fixed number big enough, and $d_n$ is a level to determine. In practice, the integer $s$ should be thought as a preliminary bound on the number of possible spikes.

\subsection{Consistency}

\begin{theorem} \label{consistance}
Let  $(\mathsf{x}_i)_{(1 \le i \le n)}$ be $n$ copies i.i.d. of $\mathsf{x}=EV^\frac{1}{2}\mathsf{y}$, where $\mathsf{y} \in \mathbb{R}^p$ is a zero-mean random vector of i.i.d. components which satisfies Assumptions \ref{hypothese2} and $E$ is an orthogonal matrix. Assume that
$$V=\mbox{cov}(\mathsf{x})= \sigma^2 \left ( \begin{array}{cc} \Sigma_{q_0} & 0 \\ 
								      0 & I_{p-q_0} \end{array} \right )$$
where $\Sigma$ has $q_0$ non null, non unit and different eigenvalues $\alpha_1 > \dots > \alpha_{q_0} > 1+\sqrt{c}$. Assume that $\frac{p}{n} \rightarrow c >0$ when $n \rightarrow +\infty$.\\
Let $(d_n)_{n \ge 0}$ be a real sequence such that $d_n \rightarrow 0$ and $n^{2/3}d_n \rightarrow +\infty$. Then the estimator $\hat{q}_n$ is strongly consistent, i.e $\widehat{q}_n \rightarrow q_0$ almost surely when $n \rightarrow +\infty$.
\end{theorem}

In the sequel, we will assume that $\sigma^2=1$ (If it is not the case, we consider $\frac{\lambda_{n,j}}{\sigma^2}$). For the proof, we need two theorems. The first, Proposition \ref{Paul}, shows that the limiting law of $\lambda_{n,j} - \phi(\alpha_j)$ is Gaussian (Bai and Yao \cite{Bai-Yao} and Paul \cite{Paul}):

\begin{proposition} \label{Paul}
Assume that the entries $\mathsf{x}^i$ of $\mathsf{x}$ satisfy $\mathbb{E}(\|\mathsf{x}^i\|^4) < +\infty$, $\alpha_j~>~1+\sqrt{c}$ for all $1 \le j \le q_0$ and have multiplicity 1. Then as $p$, $n \rightarrow +\infty$ so that $\frac{p}{n} \rightarrow c$,
$$\sqrt{n}(\lambda_{n,j} - \phi(\alpha_j)) \overset{\mathcal{L}}{\longrightarrow} \mathcal{N}(0,\sigma^2(\alpha_j))$$
where $\sigma^2(\alpha_j) = 2\alpha_j^2 \left ( 1-\frac{c}{(\alpha_j-1)^2}\right )$.
\end{proposition}

The second Proposition \ref{Maida} is issued from the Proposition 5.8 of \cite{Maida}:

\begin{proposition} \label{Maida}
Under the Assumptions \ref{hypothese2}, for all $1 \le i \le L$ with a prefixed range $L$,
$$\frac{n^{\frac{2}{3}}}{\beta}(\lambda_{n,q_0+i}-b) = O_{\mathbb{P}}(1)\text{,}$$
where $\beta=(1+\sqrt{c})(1+\sqrt{c^{-1}})^{\frac{1}{3}}\text{.}$
\end{proposition}

We also need the following lemma:
\begin{lemma} \label{lemmatight}
Let $(X_n)_{n \ge 0}$ be a tight sequence of random variables. Then for all real sequence $(u_n)_{n \ge 0}$ which diverges to infinity,
$$\mathbb{P} (|X_n| \ge u_n) \rightarrow 0\text{.}$$ 
\end{lemma}

\begin{proof}
As $(X_n)_{n \ge 0}$ is a tight sequence, for all $\varepsilon >0$,  it exists a compact $K$ such that, for all $n \in \mathbb{N}$, $\mathbb{P}(\mathsf{X}_n \notin K) < \varepsilon$.
Furthermore, as $u_n \rightarrow +\infty$, it exists $n \in \mathbb{N}$ such that for all $n \ge N$, $[-u_n,u_n] \supset K$. So $\mathbb{P}(|\mathsf{X}_n| > u_n) \le \mathbb{P}(\mathsf{X}_n \notin K) < \varepsilon$.
Consequently, $\mathbb{P}(|\mathsf{X}_n| > u_n) \rightarrow 0$.
\end{proof}

\begin{proof} of Theorem \ref{consistance}.
We have
\begin{eqnarray*}
\{ \hat{q}_n = q_0 \} & = & \{ q_0=\mbox{max}\{j : \delta_j \ge d_n\} \} \\
 & = & \{ \forall j \in \{1,\dots,q_0\}\mbox{, }\delta_{n,j} \ge d_n \} \cap \{\delta_{n,q_0+1} < d_n\}\text{.}\\
 \end{eqnarray*}
Therefore
\begin{eqnarray*}
\mathbb{P}(\hat{q}_n = q_0) & = & \mathbb{P}\left ( \bigcap_{1\le j \le q_0} \{ \delta_{n,j} \ge d_n \} \cap \{\delta_{n,q_0+1} < d_n\} \right )\\
& = & 1 - \mathbb{P}\left ( \bigcup_{1\le j \le q_0} \{ \delta_{n,j} < d_n \} \cup \{\delta_{n,q_0+1} \ge d_n\} \right )\\
& \ge & 1 - \sum_{j=1}^{q_0} \mathbb{P}(\delta_{n,j} < d_n)-\mathbb{P}(\delta_{n,q_0+1} \ge d_n ) \text{.} \\
\end{eqnarray*}

\noindent \emph{Case of $j=q_0+1$}.
In this case, $\delta_{n,q_0+1} = \lambda_{n,q_0+1}-\lambda_{n,q_0+2}$ (non-spike eigenvalues). We consider the following sequence of random variables
$$Y_n=\frac{n^{\frac{2}{3}}}{\beta}(\lambda_{n,q_0+i}-b)\text{.}$$
By Proposition \ref{Maida}, $(Y_n)_{n \ge 1}$ is a tight sequence. So by using Lemma \ref{lemmatight}, for any sequence $(a_n)_{n \ge 0}$, $a_n \rightarrow +\infty$ we have
$$\mathbb{P} (|Y_n| \ge a_n) \rightarrow 0\text{.}$$
Therefore

\begin{eqnarray*}
\mathbb{P}(|Y_n| \le a_n) & = & \mathbb{P}\left (\frac{n^{\frac{2}{3}}}{\beta}(|lambda_{n,q_0+i}-b| \le a_n\right )\\
& = & \mathbb{P} \left ( |\lambda_{n,q_0+i}-b| \le \frac{a_n}{n^{\frac{2}{3}}} \beta \right )\\
&~& \longrightarrow 1 \text{.}
\end{eqnarray*}

We choose $d_n\rightarrow 0$ such that $n^{2/3} d_n \rightarrow +\infty$. So we have
$$\mathbb{P}(\lambda_{n,q_0+i} \in \mathsf{J}_{n}) \rightarrow 1\text{,}$$
with
$$\mathsf{J}_{n}= \left [  b \pm d_n  \right ]\text{.}$$
It follows
$$\mathbb{P}  \left( \delta_{n,q_0+1} \le d_n \right) \ge \mathbb{P} \left ( \{\lambda_{n,q_0+i} \in \mathsf{J}_{n}\}\cap\{\lambda_{n,q_0+i+1} \in \mathsf{J}_{n}\} \right ) \rightarrow 1\text{.}$$
Therefore
$$\mathbb{P}(\delta_{n,q_0+1} \ge d_n) \rightarrow 0$$

\noindent \emph{Case of $1 \le j \le q_0$}. These indices correspond to the spike eigenvalues. By using Proposition \ref{Paul} and the previous argument, it is easy to show that we can choose a real sequence $(b_n)_{n \ge 0}$, $b_n \rightarrow 0$ such that $\sqrt{n}b_n \rightarrow +\infty$ and
$$\mathbb{P}(\lambda_{n,j} \in \mathsf{I}_{n,j}) \rightarrow 1\text{,}$$
where
$$\mathsf{I}_{n,j} = \left [  \phi(\alpha_j) \pm b_n \right ]\text{.}$$
Therefore
\begin{itemize}
\item For all $1 \le j < q_0$, we have
$$\mathbb{P} \left (\delta_{n,j} \ge \phi(\alpha_j) -\phi(\alpha_{j+1}) - b_n \right ) \ge \mathbb{P} \left ( \{\lambda_{n,j} \in \mathsf{I}_{n,j}\}\cap\{\lambda_{n,j+1} \in \mathsf{I}_{n,j+1}\} \right ) \rightarrow 1\text{.}$$

Let
$$c_{n,j}=\phi(\alpha_j) -\phi(\alpha_{j+1}) -b_n\text{.}$$

\item For $j=q_0$, $\delta_{n,q_0}=\lambda_{n,q_0}-\lambda_{n,q_0+1}$. By using the first section of the proof, one can show that
$$\mathbb{P} \left (\delta_{n,q_0} \ge \phi(\alpha_{q_0}) -b - \left ( b_n+d_n \right ) \right ) \ge \mathbb{P} \left ( \{\lambda_{n,q_0} \in \mathsf{I}_{n,q_0}\}\cap\{\lambda_{n,q_0+1} \in \mathsf{J}_{n}\} \right ) \rightarrow 1\text{.}$$

Let
$$c_{n,q_0}=\phi(\alpha_{q_0}) -b - \left ( b_n + d_n \right )\text{.}$$

\item
Therefore for all $0 \le j \le q_0$ we have
\begin{eqnarray*}
\mathbb{P}(\delta_{n,j} \ge c_{n,j}) \rightarrow 1 & \Rightarrow & \mathbb{P}(\delta_{n,j} < c_{n,j})  \rightarrow 0\text{.}\\
\end{eqnarray*}
\end{itemize}

As $d_n \rightarrow 0$ and for all $1 \le j \le q_0$, $c_{n,j} \rightarrow c_j > 0$, it exists $N \in \mathbb{N}:\forall n \ge N$,
$$\mathbb{P}(\delta_{n,j} < d_n) \le \mathbb{P}(\delta_{n,j} < c_{n,j}) \rightarrow 0\text{.}$$
So we have
$$\sum_{j=1}^{q_0} \mathbb{P}( \delta_{n,j} < d_n) \rightarrow 0\text{.}$$

\noindent \emph{Conclusion}.
$\mathbb{P}(\delta_{n,q_0+1} \ge d_n)  \rightarrow 0$ and $\sum_{j=1}^{q_0} \mathbb{P}( \delta_{n,j} < d_n) \rightarrow 0$, therefore

$$\mathbb{P}(\hat{q}_n = q_0) \underset{n \rightarrow +\infty}{\longrightarrow}1\text{.}$$ 
\end{proof}

\subsection{Simulation experiments}

Now we will illustrate the previous result by some simulations. First, we have to chose the sequence $d_n$ to be used. Theoretically speaking, all the sequences satisfying the requirement $d_n\rightarrow 0$ such that $n^{2/3}d_n \rightarrow +\infty$ are convenient. We tested several sequences and we decided to take one of the form $\frac{a_n}{n^{2/3}}\beta$ which a sequence $(a_n)_{n\ge 0}$ proportional to $\sqrt{2\log\log n}$ : this idea came from that, as in the case of the mean of i.i.d random variables, the $\lambda_{n,j}$ corresponding to the non-spikes tend to a gaussian law. So we can conjecture a result analog to the law of the iterated logarithm for the $\lambda_{n,j}$, $j>q_0$. Finally, we choose $a_n=4\sqrt{2\log\log n}$ and simulate two different models: one with dispersed spikes which should lead to an easier estimation of $q_0$, and a more difficult case with closer spikes:
\vspace{0.5cm}
\begin{itemlist}
\item {\bf Model 1}: $q_0=5$, $(\alpha_1,\alpha_2,\alpha_3,\alpha_4,\alpha_5)=(259.72,17.97,11.04,7.88,4.82)$;
\item {\bf Model 2}: $q_0=4$, $(\alpha_1,\alpha_2,\alpha_3,\alpha_4)=(7,6,5,4)$.
\end{itemlist}
\vspace{0.5cm}
Note that the values of Model 1 have been chosen to be the same as in \cite{Harding}. For each model, two different values of $c$, 0.3 and 0.6, are considered. We give in Tables 1-2 and 3, respectively, the distribution of $\hat{q}_n$, its mean and mean squared error over 1000 independent replications. The frequency of $\hat{q}_n=q_0$ is given in Figure 1.

\begin{table}[!hb]
\tbl{Mean, mean squared error and empirical distribution of $\hat{q}_n$ over 1000 independent replications for Model 1.}
{\begin{tabular}{|ccc|ccccccc|}
\cline{4-10}
 \multicolumn{3}{c}{} &\multicolumn{7}{|c|}{Distribution of $\hat{q}_n$}\\
\hline
$(p,n)$ & Mean & MSE & 1 &  2  &  3  &  4  &  \textbf{5}  &  6 & 7 \\
\hline
(30,100) & 5.057 & 0.212 & 0.001 & 0.007 & 0.009 & 0.0\hphantom{00} & \textbf{0.883} & 0.1\hphantom{00} & 0.002\\
(60,200) &  5.081 & 0.107 & 0.001 & 0.001 & 0.0\hphantom{00} & 0.0\hphantom{00} & \textbf{0.91\hphantom{0}} & 0.088 & 0.0\hphantom{00}\\
(120,400) & 5.079 & 0.073 & 0.0\hphantom{00} & 0.0\hphantom{00} & 0.0\hphantom{00} & 0.0\hphantom{00} & \textbf{0.921} & 0.079 & 0.0\hphantom{00}\\
(240,800) & 5.069 & 0.064 & 0.0\hphantom{00} & 0.0\hphantom{00} & 0.0\hphantom{00} & 0.0\hphantom{00} & \textbf{0.931} & 0.069 & 0.0\hphantom{00}\\
\hline
\end{tabular}}
\end{table}

\begin{table}[!ht]
\tbl{(Continued) Mean, mean squared error and empirical distribution of $\hat{q}_n$ over 1000 independent replications for Model 1.}
{\begin{tabular}{|ccc|ccccccc|}
\cline{4-10}
 \multicolumn{3}{c}{} &\multicolumn{7}{|c|}{Distribution of $\hat{q}_n$}\\
\hline
$(p,n)$ & Mean & MSE & 1 &  2  &  3  &  4  &  \textbf{5}  &  6 & 7 \\
\hline
(60,100) & 5.056 & 0.139 & 0.001 & 0.004 & 0.003 & 0.002 & \textbf{0.914} & 0.076 & 0.0\hphantom{00}\\
(120,200) & 5.08\hphantom{0} & 0.098 & 0.0\hphantom{00} & 0.001 & 0.002 & 0.0\hphantom{00} & \textbf{0.91\hphantom{0}} & 0.087 & 0.0\hphantom{00}\\
(240,400) & 5.072 & 0.079 & 0.002 & 0.0\hphantom{00} & 0.0\hphantom{00} & 0.0\hphantom{00} & \textbf{0.924} & 0.075 & 0.0\hphantom{00}\\
(480,800) & 5.072 & 0.069 & 0.0\hphantom{00} & 0.0\hphantom{00} & 0.0\hphantom{00} & 0.0\hphantom{00} & \textbf{0.929} & 0.07\hphantom{0} & 0.001\\
\hline
\end{tabular}}
\end{table}

\begin{table}[!ht]
\tbl{Mean, mean squared error and empirical distribution of $\hat{q}_n$ over 1000 independent replications for Model 2.}
{\begin{tabular}{|ccc|cccccc|}
\cline{4-9}
 \multicolumn{3}{c}{} &\multicolumn{6}{|c|}{Distribution of $\hat{q}_n$}\\
\hline
$(p,n)$ & Mean & MSE & 0 & 1 &  2  & 3  &  \textbf{4}  &  5  \\
\hline
(30,100) & 3.718  &  1.086 & 0.0\hphantom{00} & 0.001 & 0.059 & 0.0\hphantom{00} & \textbf{0.778} & 0.085\\
(60,200) &  3.925  &  0.582  & 0.013 & 0.024 & 0.019 & 0.0\hphantom{00} & \textbf{0.857} & 0.087\\
(120,400) & 4.005  &  0.331  & 0.01\hphantom{0} & 0.01\hphantom{0} & 0.001 & 0.0\hphantom{00} & \textbf{0.902} & 0.077\\
(240,800) & 4.062  &  0.110 &  0.002 & 0.001 & 0.0\hphantom{00} & 0.0\hphantom{00} & \textbf{0.924} & 0.073\\
\hline
(60,100) & 3.478 & 1.655 & 0.053 & 0.086 & 0.059 & 0.001 & \textbf{0.734} & 0.067\\
(120,200) & 3.818 & 0.823 & 0.025 & 0.033 & 0.024 & 0.0\hphantom{00} & \textbf{0.853} & 0.065\\
(240,400) & 3.969 & 0.394 & 0.009 & 0.015 & 0.011 & 0.0\hphantom{00} & \textbf{0.893} & 0.072\\
(480,800) & 4.051  &  0.108 & 0.003 & 0.0\hphantom{00} & 0.0\hphantom{00} & 0.0\hphantom{00} & \textbf{0.934} & 0.063\\
\hline
\end{tabular}}
\end{table}

\begin{figure}[!ht]
 \begin{center}
   \includegraphics{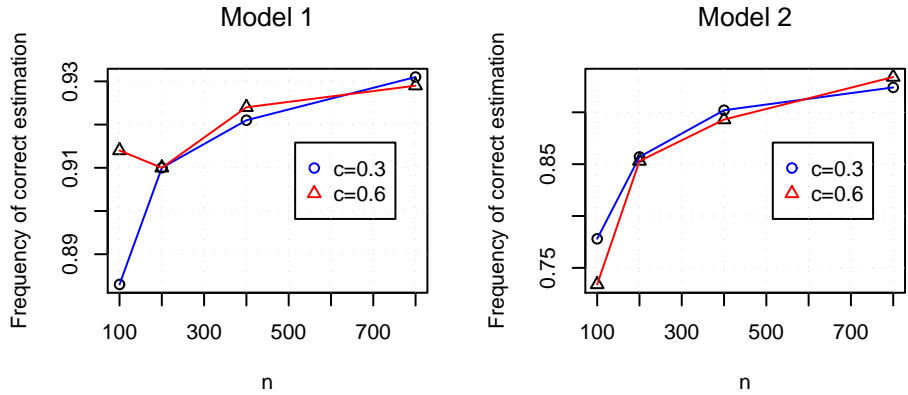}
   \caption{Frequency of $\hat{q}_n=q_0$ over 1000 independent replications.\label{1}}
 \end{center}
\end{figure}

In both cases, we can observe the asymptotic consistency of the estimator. Comparing the two models, except the last case $(p,n)=(480,800)$, the estimator performs better in Model 1 than in Model 2. This phenomenon is due to the fact that the differences between consecutive eigenvalues are smaller than in Model 2 so that it is more difficult to distinguish spikes from non spikes.\\

\noindent Within a same model, the convergence is slower in the $c=0.6$ case. We could explain this by the fact that the gap between two consecutive spike eigenvalues stays the same, and when $c$ increases, the spectrum of $\mathsf{S}_n$ is more dispersed, so that the differences $\delta_{n,j}$ from non- spikes are larger and again our detection problem is more difficult.\\

\noindent It is worth mentioning that the chosen constant $d_n=\frac{4\sqrt{2\log\log n}}{n^{2/3}}\beta$ leads to a slight  over-estimation of $q_0$ for the tested sizes $(p,n)$. This finite-sample behaviour could be improved with a more sophisticated choice of $d_n$ which however seems a difficult point to address.\\

\section{Case of Simple Spikes with Unknown Variance}

In practice, the scale parameter $\sigma^2$ is also unknown and we need to estimate it as well. First, we will explain how to do in the non-spikes (null) case, i.e. $V=\sigma^2 I_p$, and then in the case with spikes.

\subsection{Estimation of the variance in the white case}

We consider a zero-mean random vector $\mathsf{x} \in \mathbb{R}^p$ with population covariance matrix
$$V=\mbox{cov}(\mathsf{x})= \sigma^2 I_p\text{.}$$
We keep the previous assumptions. We will use the law of large numbers to estimate the unknown variance $\sigma^2$. We have the following theorem (Mar{\v c}enko and Pastur~\cite{MP}, Bai and Silverstein \cite{BS})

\begin{proposition} \label{estimation}
  Assume that, for any $\eta \ge 0$ :
  $$\frac{1}{\eta^2 np} \sum_{j,k} \mathbb{E}(|x_jk|^2 \mathds{1}_{|x_jk| \ge \eta \sqrt{n}}) \rightarrow 0 \mbox{ when } n \rightarrow +\infty\text{.}$$
  Then, with probability one, the empirical spectral distribution (ESD) $F^{S_n}$ of $S_n$ weakly converges to the Mar{\v c}enko-Pastur distribution
  with ratio index $c$ and scale parameter $\sigma^2$, denoted by $F^{c,\sigma^2}(x)$, which has a density function
  $$p_{c,\sigma^2}(x)=\left \{  \begin {array}{cl} \frac{1}{2\pi x c \sigma^2} \sqrt{(b^{+}-x)(x-b^{-})} & \mbox{ if } b^{-} \le x \le b^{+}\\
0 & \mbox{ otherwise } \end{array} \right.\text{,}$$
where $b^{-}=\sigma^2(1-\sqrt{c})^2$ and $b^{+}=\sigma^2(1+\sqrt{c})^2$.
\end{proposition}
Note that $\sigma^2$ represents the mean of the limiting distribution. Moreover, it is well-known that under the condition of the Proposition \ref{estimation}, it holds almost surely,
$$\widehat{\sigma}^2=\frac{1}{p} \sum_{i=1}^p \lambda_{n,i} \rightarrow \sigma^2 \text{.}$$

\subsection{Determining the number of spikes with an unknown variance.}

As we notice in the first section, when the variance is known and different of one, we only have to divide the consecutive difference $\delta_{i,n}$ by this variance. As the variance is unknown, we will replace it by the estimate $\widehat{\sigma}^2=\frac{1}{p} \sum_{i=1}^p \lambda_{n,i}$, which converges almost surely to $\sigma^2$ when $p \rightarrow +\infty$. Nevertheless, because of the spikes, the variance of $\widehat{\sigma}^2$ will be greater than the one in the null case. The variance will be minimum if we only take the mean of the non-spike eigenvalues i.e those that have an index $i \ge q_0+1$. The problem is that we don't know $q_0$. By consequence, the idea is to make a first estimation $\hat{q}_n^0$ of $q_0$ with $\widehat{\sigma}^2_0=\frac{1}{p} \sum_{i=1}^p \lambda_{n,i}$. Then, if $\hat{q}_n^0 >0$, we set $\widehat{\sigma}^2_1=\frac{1}{p-\hat{q}_n^0} \sum_{i=\hat{q}_n^0+1}^p \lambda_{n,i}$ (So we have $\widehat{\sigma}_0^2 \ge \widehat{\sigma}_1^2$), and we reestimate $q_0$ by $\hat{q}_n^1$ using this new estimation. We repeat it until we find an indice $k$ such that $\hat{q}_n^k=\hat{q}_n^{k+1}$. If such an indice doesn't exist, the algorithm will stop at the preliminary bound $k=s$ fixed initially. To sum up, here is the algorithm:

\begin{verbatim}
q1=0
sigma2=1/p*(lambda_1+...+lambda_p)
q2="estimator of the known variance case with division by sigma2"

while q2~=q1 do
 q1:=q2
 sigma2=1/(p-q1)*(lambda_(q1+1)+...+lambda_p
 q2="estimator of the known variance case with division by sigma2"
end

result=(q1,sigma2)
\end{verbatim}

\subsection{Simulation experiments}

We conduct the simulations with two values of the variance $\sigma^2=1$, and $\sigma^2=500$ to see if a high variance will influence the estimation. We keep the same other parameters as in the previous simulation study of Section 3 and estimate $\sigma^2$ and the number of spikes with the method explained above. Additional to the statistics about the spikes number estimator $\hat{q}_n$, we provide also those about the final estimate $\widehat{\sigma}^2$ of the unknown variance. The results are displayed in Tables 4 to 8.

\begin{table}[!hb]
\tbl{Mean, mean squared error and empirical distribution of $\hat{q}_n$, mean and mean squared error of $\widehat{\sigma}^2$ over 1000 independent replications for Model 1 and $\sigma^2=1$.}
{\begin{tabular}{|ccc|ccccccc|cc|}
 \cline{4-12}
 \multicolumn{3}{c}{}& \multicolumn{7}{|c|}{Distribution of $\hat{q}_n$} & \multicolumn{2}{c|}{$\widehat{\sigma}^2$}\\
\hline
$(p,n)$ & Mean & MSE & 1 & 2 & 3 & 4 & \textbf{5} & 6 & 7 & Mean & MSE\\
\hline
(30,100) & 5.052 & 0.338 & 0.003 & 0.015 & 0.008 & 0.0\hphantom{00} & \textbf{0.849} & 0.0\hphantom{00} & 0.125 & 0.955 & 0.015\\
(60,200) & 5.108 & 0.112 & 0.0\hphantom{00} & 0.001 & 0.0\hphantom{00} & 0.0\hphantom{00} & \textbf{0.89\hphantom{0}} & 0.107 & 0.002 & 0.97\hphantom{0} & 0.0\hphantom{00}\\
(120,400) & 5.069 & 0.076 & 0.0\hphantom{00} & 0.001 & 0.0\hphantom{00} & 0.0\hphantom{00} & \textbf{0.927} & 0.072 & 0.0\hphantom{00} & 0.986 & 0.0\hphantom{00}\\
(240,800) & 5.084 & 0.077 & 0.0\hphantom{00} & 0.0\hphantom{00} & 0.0\hphantom{00} & 0.0\hphantom{00} & \textbf{0.916} & 0.084 & 0.0\hphantom{00} & 0.993 & 0.0\hphantom{00}\\
\hline
\end{tabular}}
\end{table}

\begin{table}[!ht]
\tbl{(Continued) Mean, mean squared error and empirical distribution of $\hat{q}_n$, mean and mean squared error of $\widehat{\sigma}^2$ over 1000 independent replications for Model 1 and $\sigma^2=1$.}
{\begin{tabular}{|ccc|ccccccc|cc|}
 \cline{4-12}
 \multicolumn{3}{c}{}& \multicolumn{7}{|c|}{Distribution of $\hat{q}_n$} & \multicolumn{2}{c|}{$\widehat{\sigma}^2$}\\
\hline
(60,100) & 5.087 & 0.236 & 0.001 & 0.009 & 0.004 & 0.0\hphantom{00} & \textbf{0.865} & 0.122 & 0.002 & 0.943 & 0.003\\
(120,200) & 5.095 & 0.092 & 0.0\hphantom{00} & 0.0\hphantom{00} & 0.001 & 0.0\hphantom{00} & \textbf{0.902} & 0.097 & 0.0\hphantom{00} & 0.971 & 0.0\hphantom{00}\\
(240,400) &  5.07 & 0.065 & 0.0\hphantom{00} & 0.0\hphantom{00} & 0.0\hphantom{00} & 0.0\hphantom{00} & \textbf{0.93\hphantom{0}} & 0.07\hphantom{0} & 0.0\hphantom{00} & 0.985 & 0.0\hphantom{00}\\
(480,800) & 5.067 & 0.063 & 0.0\hphantom{00} & 0.0\hphantom{00} & 0.0\hphantom{00} & 0.0\hphantom{00} & \textbf{0.933} & 0.067 & 0.0\hphantom{00} & 0.993 & 0.0\hphantom{00}\\
\hline
\end{tabular}}
\end{table}

\begin{table}[!ht]
\tbl{Mean, mean squared error and empirical distribution of $\hat{q}_n$, mean and mean squared error of $\widehat{\sigma}^2$ over 1000 independent replications for Model 2 and $\sigma^2=1$.}
{\begin{tabular}{|ccc|ccccccc|cc|}
 \cline{4-12}
 \multicolumn{3}{c}{}& \multicolumn{7}{|c|}{Distribution of $\hat{q}_n$} & \multicolumn{2}{c|}{$\widehat{\sigma}^2$}\\
\hline
$(p,n)$ & Mean & MSE & 0 & 1 & 2 & 3 & \textbf{4} & 5 & 6 & Mean & MSE\\
\hline
(30,100) & 3.362  &  2.019 & 0.079 & 0.078 & 0.091 & 0.0\hphantom{00} & \textbf{0.658} & 0.094 & 0.0\hphantom{00} & 1.052 & 0.043\\
(60,200) &  3.806  &  1.023&  0.032 & 0.038 & 0.026 & 0.0\hphantom{00} & \textbf{0.805} & 0.098 & 0.001 &  0.994 & 0.005 \\
(120,400) & 3.983  &  0.483 & 0.019 & 0.008 & 0.004 & 0.0\hphantom{00} & \textbf{0.878} & 0.091 & 0.0\hphantom{00} & 0.991 & 0.001\\
(240,800)& 4.071  &  0.144 & 0.003 & 0.001 & 0.001 & 0.0\hphantom{00} & \textbf{0.907} & 0.088 & 0.0\hphantom{00} & 0.994 & 0.0\hphantom{00}\\
\hline
(60,100) & 3.367 & 1.898& 0.069 & 0.081 & 0.096 & 0.001 & \textbf{0.674} & 0.079 & 0.0\hphantom{00} & 1.003 & 0.012\\
(120,200)& 3.781 & 1.04\hphantom{0} & 0.034 & 0.034 & 0.036 & 0.0\hphantom{00} & \textbf{0.806} & 0.089 & 0.001 & 0.986 & 0.002\\
(240,400) & 3.965 & 0.472 & 0.015 & 0.015 & 0.007 & 0.0\hphantom{00} & \textbf{0.892} & 0.071 & 0.0\hphantom{00} & 0.99\hphantom{0} & 0.0\hphantom{00}\\
(480,800) & 4.052  &  0.125 & 0.002 & 0.003 & 0.0\hphantom{00} & 0.0\hphantom{00} & \textbf{0.926} & 0.069 & 0.0\hphantom{00} & 0.994 & 0.0\hphantom{00}\\
\hline
\end{tabular}}
\end{table}

\begin{figure}[!hb]
   \includegraphics{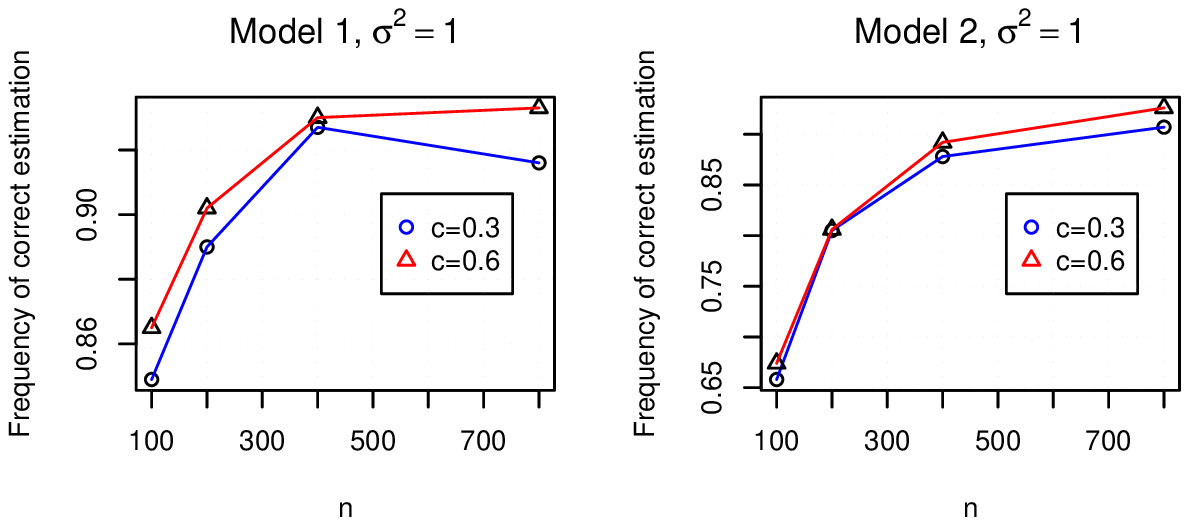}
   \caption{Frequency of $\hat{q}_n=q_0$ over 1000 independent replications with $\sigma^2=1$.}
\end{figure}

\begin{table}[!ht]
\tbl{Empirical distribution of $\hat{q}_n$, mean and mean squared error of $\widehat{\sigma}^2$ over 1000 independent replications for Model 1 and $\sigma^2=500$.} 
{\begin{tabular}{|c|ccccccc|cc|}
 \cline{2-10}
 \multicolumn{1}{c}{}& \multicolumn{7}{|c|}{Distribution of $\hat{q}_n$} & \multicolumn{2}{c|}{$\widehat{\sigma}^2$}\\
\hline
$(p,n)$ & 1 & 2 & 3 & 4 & \textbf{5} & 6 & 7 & Mean & MSE\\
\hline
(30,100) & 0.003 & 0.012 & 0.005 & 0.0\hphantom{00} & \textbf{0.823} & 0.155 & 0.002 & 474.909 & 3281.714\\
(60,200) &  0.0\hphantom{00}  & 0.001 & 0.0\hphantom{00} & 0.0\hphantom{00} & \textbf{0.904} & 0.094 & 0.001 & 485.019 & \hphantom{00}99.558\\
(120,400) & 0.0\hphantom{00}  & 0.001 & 0.0\hphantom{00} & 0.0\hphantom{00} & \textbf{0.918} & 0.080 & 0.001 & 492.608 & \hphantom{00}21.244\\
(240,800) & 0.0\hphantom{00} & 0.0\hphantom{00} & 0.0\hphantom{00} & 0.0\hphantom{00} & \textbf{0.914} & 0.086 & 0.0\hphantom{00} & 496.316 & \hphantom{000}3.519\\
\hline
(60,100) & 0.002 & 0.008 & 0.006 & 0.001 & \textbf{0.870} & 0.113 & 0.0\hphantom{00} & 472.816 & 688.994 \\
(120,200) & 0.0\hphantom{00} & 0.002 & 0.0\hphantom{00} & 0.0\hphantom{00} & \textbf{0.898} & 0.099 & 0.001  & 485.49\hphantom{0} & \hphantom{0}55.489  \\
(240,400) & 0.0\hphantom{00}  & 0.0\hphantom{00} & 0.0\hphantom{00} & 0.0\hphantom{00} & \textbf{0.928} & 0.071 & 0.001  & 492.699 & \hphantom{00}7.242\\
(480,800) & 0.0\hphantom{00} & 0.0\hphantom{00} & 0.0\hphantom{00} & 0.0\hphantom{00} & \textbf{0.933} & 0.067 & 0.0\hphantom{00} & 496.377 & \hphantom{00}1.654\\
\hline
\end{tabular}}
\end{table}

\begin{table}[!ht]
\tbl{Empirical distribution of $\hat{q}_n$, mean and mean squared error of $\widehat{\sigma}^2$ over 1000 independent replications for Model 2 and $\sigma^2=500$.}
{\begin{tabular}{|c|ccccccc|cc|}
 \cline{2-10}
 \multicolumn{1}{c}{}& \multicolumn{7}{|c|}{Distribution of $\hat{q}_n$} & \multicolumn{2}{c|}{$\widehat{\sigma}^2$}\\
\hline
$(p,n)$ & 0 & 1 & 2 & 3 & \textbf{4} & 5 & 6 & Mean & MSE\\
\hline
(30,100) &  0.079 & 0.088 & 0.090 & 0.0\hphantom{00} & \textbf{0.649} & 0.093 & 0.001 & 528.651 & 11223.872\\
(60,200) &  0.037 & 0.037 & 0.029 & 0.0\hphantom{00} & \textbf{0.794} & 0.103 & 0.0\hphantom{00} &  498.032 & \hphantom{0}1478.184 \\
(120,400) & 0.009 & 0.01\hphantom{0} & 0.005 & 0.0\hphantom{00} & \textbf{0.880} & 0.096 & 0.0\hphantom{00} & 494.613 & \hphantom{00}107.355\\
(240,800) & 0.003 & 0.0\hphantom{00} & 0.002 & 0.0\hphantom{00} & \textbf{0.918} & 0.075 & 0.002 & 496.813 & \hphantom{0000}8.770\\
\hline
(60,100) & 0.071 & 0.104 & 0.059 & 0.001 & \textbf{0.687} & 0.078 & 0 & 501.754  & 3126.083\\
(120,200) & 0.036 & 0.038 & 0.043 & 0.0\hphantom{00} & \textbf{0.809} & 0.074 & 0.0\hphantom{00} & 493.687 & \hphantom{0}438.063\\
(240,400) & 0.013 & 0.007 & 0.009 & 0.0\hphantom{00} & \textbf{0.900} & 0.071 & 0.0\hphantom{00} & 494.445 & \hphantom{00}39.686\\
(480,800) & 0.004 & 0.001 & 0.0\hphantom{00} & 0.0\hphantom{00} & \textbf{0.941} & 0.054 & 0.0\hphantom{00} & 496.836 & \hphantom{000}3.576\\
\hline
\end{tabular}}
\end{table}

\begin{figure}[!ht]
   \includegraphics{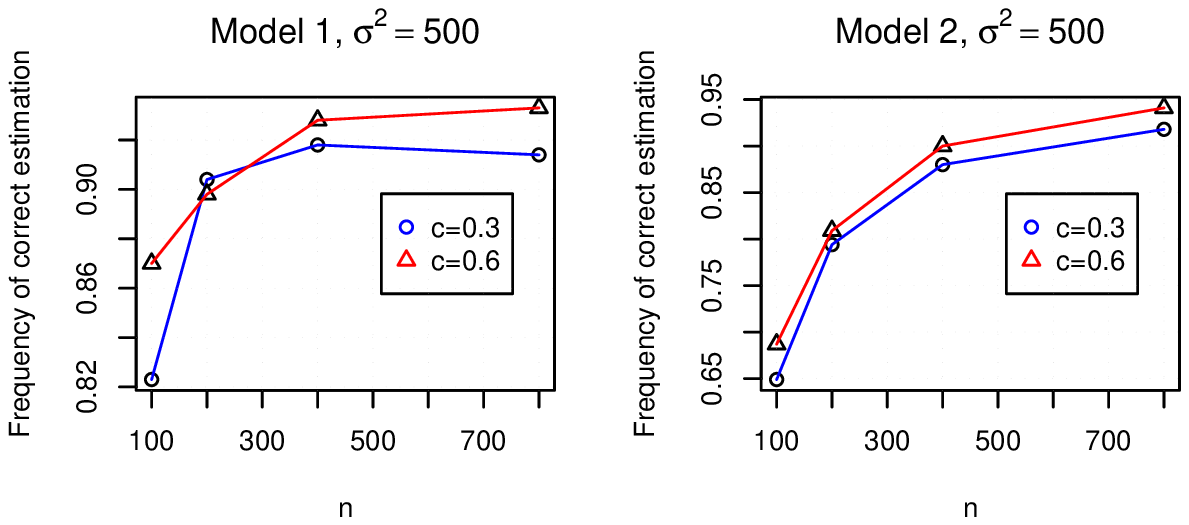}
   \caption{Frequency of $\hat{q}_n=q_0$ over 1000 independent replications with $\sigma^2=500$.}
\end{figure}

\begin{figure}[!ht]
   \includegraphics{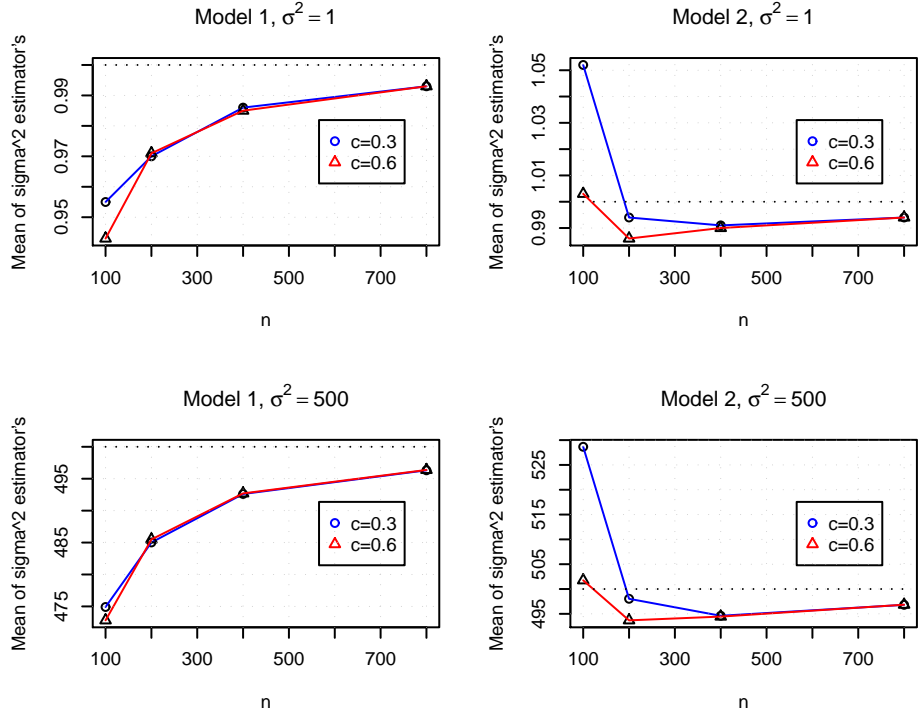}
   \caption{Mean of $\widehat{\sigma}^2$ over 1000 independent replications.}
\end{figure}

\newpage

\newpage
First, we can see the asymptotic consistancy of the estimator of $q_0$ in all the four cases. If we compare these simulations with the known variance case, we can see that the estimation is less accurate in the small $(p,n)$. Furthermore, as in the previous case, the convergence is slower in the $c=0.6$ case and the estimator performs better in Model 1 than in Model 2, for both values of $\sigma^2$. The estimation of $q_0$ is more accurate with an unknown variance of $\sigma^2=500$.\\

\noindent We also give the mean and mean squared error of $\hat{q}_n$ in the $\sigma^2=1$ case (Tables 4-5 and 6) to compare with Tables 1-2 and 3, where $\sigma^2=1$ also, to see the effect of its estimation. The variance and the bias are higher especially for small values of $(p,n)$ in this case with unknown variance.\\

\noindent The estimation of $\sigma^2$ performs well, but it seems to be underestimated. There is no particular difference between the two values of $c$ in Model 1 but in Model 2, contrary to the estimation of $q_0$, the convergence seems to be faster in the $c=0.6$ case for $\widehat{\sigma}^2$. The variance of the estimator decreases with the increase of $n$ and $p$, and is less in the $c=0.6$ case. As expected, the variance is lower in the $\sigma^2=1$ case.

\section{Comparaison with two Related Methods}

In signal processing or econometric literature, the factor model (or linear mixture model) is often used. This model is defined as follows: let $(\mathsf{x}_i = \mathsf{x}(t_i))_{(1 \le i \le n)}$ be an i.i.d $n$-sample of $p$-dimensional random vectors satisfying
 \begin{eqnarray*}
 \mathsf{x}(t)&=&\sum_{k=1}^{q_0} a_k s_k(t) + \sigma u(t)\\
 &=& \mathsf{A} s(t)+ \sigma u(t)\text{,}\\
 \end{eqnarray*}
where 
\begin{itemlist}
\item $s(t)=(s_1(t),\ldots,s_{q_0}(t))' \in \mathbb{R}^{q_0}$ are $q_0$ random factors (or signals) assumed to have zero mean, unit variance and mutually uncorrelated;
\item $A=(a_1,\ldots,a_{q_0})$ is a $p \times q_0$ fixed unknown matrix of rank $q_0$ (response vectors or factor loadings);
\item $\sigma \in \mathbb{R}$ is the noise level, $u \sim \mathcal{N}(0,\mathsf{I}_p)$.
\end{itemlist}

It is easy to show that in this case, the population covariance matrix takes the form of a spiked population model: the spikes are only slightly modified. If we denote by $\alpha'$ the vector of spikes in the factor model, we have the following relationship with our original vector $\alpha$
$$\alpha=\frac{\alpha'}{\sigma^2}+1\text{.}$$
Here determining the number of spikes $q_0$ means the detection of the number of factors/signals $q_0$. We will explain and compare two methods from Econometrics (Harding \cite{Harding}) and signal processing (Kritchman \& Nadler \cite{Nadler}), respectively.

\subsection{Method of Harding and comparison}

In his paper \cite{Harding}, Harding uses less restrictive hypotheses as the sequence $(u(t))$ is not necessarily independent, but he simulates a Gaussian model. His general idea is to compare the spectral moments of $S_n$ with the ESD of $S_n$ without the factors (or spikes), and to remove the largest eigenvalues one by one in $S_n$ until a ``distance" between the moments is minimum.\\ 

More precisely, the variance of the noise is seen as a parameter $\theta$ and his idea is to write $\mathsf{S}_n =\Xi_n + \Omega_n$ ($\mbox{rank}(\Xi_n)=q_0$) as a sum of a finite rank perturbation $\Xi_n$ of the noise covariance $\Omega_n$. Let $\Pi(\mathsf{S}_n)$ be the vector of the first $s$ moments of the empirical spectral distribution (ESD) of the covariance matrix~$\mathsf{S}_n$, $\Pi(\Omega_n)$ the equivalent for $\Omega_n$ and $\Pi(\theta)$ its limit as $p$ and $n \rightarrow +\infty$, $\frac{p}{n} \rightarrow c$. Here is the procedure of Harding:
\vspace{0.5cm}
\begin{itemlist}
\item First, compute the moments $\Pi(\theta)$ of the asymptotic eigenvalue distribution of the covariance matrix of $\Omega_n$ for a large $(p,n)$ sample.
\item By Bai and Silverstein (\cite{BS}), we have that
$p \left ( \Pi(\Omega_n) - \Pi(\theta)   \right )
							\overset{\mathcal{L}}{\longrightarrow} \mathcal{N}(\Delta,W)$. Consequently, estimate $\theta$ by:
$$\hat{\theta}_0 = \mbox{argmin}_{\theta} \underbrace {\left ( \Pi(\theta)-\Pi(\mathsf{S}_n \right )'\hat{W}^{-1} \left ( \Pi(\theta)-\Pi(\mathsf{S}_n \right )}_{J(\theta)}\text{.}$$
where $\hat{W}$ is a consistent estimate of $W$, calculate by estimated $\theta$ from a first step estimation with $W=\mathsf{I}_p$.
\item Next, remove the largest eigenvalue of the spectrum of $\mathsf{S}_n$ and re-estimate the parameter $\theta$ as previously to get a new estimate $\hat{\theta}_1$.
\item  This step is repeated by progressively removing large eigenvalues and for prefixed number of times to get a sequence of estimates $\hat{\theta}_2$,  $\hat{\theta}_3$, ...etc.
\item Finally, among the minimized objective functions $J(\hat{\theta}_i)$ choose the order one which corresponds to the smallest minimized value:
$$\hat{q}_0=\mbox{argmin}_i J(\hat{\theta}_i)\text{.}$$
\end{itemlist}

Actually, we know that for $q$ fixed and $p$, $n \rightarrow +\infty$, $\Pi(\mathsf{S}_n) \rightarrow \Pi(\theta)$. So the criterion is the minimization of the variance $W=W(\theta)$: it decreases until $q_0$ (until we have removed the eigenvalues corresponding to the spikes), then it stays stable. The procedure of Harding leads to an underestimation of $q_0$, at $p$ and $n$ fixed. That is why he penalized the function $J$ with a function of type $k\hat{\theta}g(p,n)$, where $k$ is the number of eigenvalues removed, $\hat{\theta}$ is the estimated variance at the step $q$ and $g(p,n)$ is a function such that $g(p,n) \rightarrow 0$ when $p$, $n \rightarrow +\infty$. The finally proposed choice for $g$ is the following function given by Bai and Ng \cite{Bai-Ng} based on a BIC criterion 
$$g(p,n)= \left ( \frac{p+n}{pn}\right ) \ln \left ( \frac{pn}{p+n}\right )\text{.}$$

For his simulation experiment, he tested four different ``distances" but we only keep the one based on the BIC criterion which is the best. Furthermore, we don't give all cases he tested. The simulation design was a little bit different, indeed Harding does not choose the spikes directly, but he generates $s(t)$ as a Gaussian law $\mathcal{N}(0,\mathsf{I}_p)$ and $A$ in a deterministic way. We calculate the corresponding spikes and it leads to the following values:
\vspace{0.5cm}
\begin{itemlist}
\item $(p,n)=(30,100)$: 
$(\alpha_5,\alpha_4,\alpha_3,\alpha_2,\alpha_1)=(3.817,6.877,10.038,16.973,258.719)$   
\item $(p,n)=(90,100)$: $(\alpha_5,\alpha_4,\alpha_3,\alpha_2,\alpha_1)=(3.692,7.276,10.785,18.101,259.010)$
\item $(p,n)=(210,300)$: $(\alpha_5,\alpha_4,\alpha_3,\alpha_2,\alpha_1)=(3.649,7.377,10.992,18.418,259.083)$
\item $(p,n)=(250,500)$: $(\alpha_5,\alpha_4,\alpha_3,\alpha_2,\alpha_1)=(3.634,7.448,11.057,18.453,259.005)$
\end{itemlist}
\vspace{0.5cm}
Nonetheless, these cases stay very close. Below we compare his results to ours. We only give in Table 9 the mean and mean squared errors of the estimator as reported in Harding's~paper.

\begin{table}[!ht]
\tbl{Compared mean and mean squared error of our $\hat{q}_n$ and $\widehat{\sigma}^2$ and those of Harding over 5000 independent replications and $\sigma^2=1$.}
{\begin{tabular}{|c|cccc|cccc|}
\cline{2-9}
\multicolumn{1}{c}{}& \multicolumn{4}{|c|}{$\hat{q}_0$} & \multicolumn{4}{c|}{$\widehat{\sigma}^2$}\\
 \cline{2-9}
 \multicolumn{1}{c}{}& \multicolumn{2}{|c|}{Harding estimator} & \multicolumn{2}{c|}{Our estimator} & \multicolumn{2}{c}{Harding estimator} & \multicolumn{2}{|c|}{Our estimator}\\
\hline
$(p,n)$ & Mean & MSE & Mean & MSE & Mean & MSE & Mean & MSE\\
\hline
(30,100) & 5.028 & 0.028 & 5.087 & 0.266 & 0.942 & 0.004 & 0.946 & 0.008\\
(90,100) &  5.040 & 0.048 & 5.049 & 0.232 & 0.944 & 0.001 & 0.943 & 0.0\hphantom{00}\\
(210,300) & 5.004 & 0.004 & 5.087 & 0.082 & 0.982 & 0.0\hphantom{00} & 0.980 & 0.0\hphantom{00}\\
(250,500) & 5.002 & 0.002 & 5.077 & 0.072 & 0.989 & 0.0\hphantom{00} & 0.988 & 0.0\hphantom{00}\\
\hline
\end{tabular}}
\end{table}

Both methods perform well and their results are overall very close except that Harding's estimation yields a slightly smaller MSE for $\hat{q}_0$. However, one should have in mind that this estimation has a very complex construction and a rigorous justification of its different steps is still open. Moreover, the spikes in Table 9 are large and well-separated one from another; it remains unclear how this method will perform in a case where the spikes are much smaller and close like in Model~2, considered in Sections 3 and 4. By contrast, our estimator has a very simple construction and we proved its consistency under reasonable assumptions.

\subsection{Method of Kritchman \& Nadler and comparison}

These authors assume the Gaussian case. In the absence of spikes, $n\mathsf{S}_n$ follows a Wishart distribution with parameters $n,p$. In this case, Johnstone \cite{Johnstone} gave the asymptotic distribution of the largest eigenvalue of $\mathsf{S}_n$.

\begin{proposition}
Let $\mathsf{S}_n$ be the sample covariance matrix of $n$ vectors distributed as $\mathcal{N}(0,\sigma^2 \mathsf{I}_p)$, and $\lambda_{n,1} \ge \lambda_{n,2} \ge \dots \ge \lambda_{n,p}$ be its eigenvalues. Then, when $n~\rightarrow~+\infty$, such that $\frac{p}{n} \rightarrow c > 0$
$$\mathbb{P} \left ( \frac{\lambda_{n,i}}{\sigma^2} < \frac{\beta_{n,p}}{n^{2/3}}s+b \right ) \rightarrow F_i(s)\mbox{, }s>0$$
where $b=(1+\sqrt{c})^2$, $\beta_{n,p}=\left(1+\sqrt{\frac{p}{n}}\right)\left(1+\sqrt{\frac{n}{p}}\right)^{\frac{1}{3}}$ and $F_i$ is the i-th Tracy-Widom distribution.
\end{proposition}

Assuming the variance $\sigma^2$ is known. To distinguish a spike eigenvalue $\lambda$ from a non-spike one at an asymptotic significance level $\gamma$, their idea is to check whether
\begin{eqnarray} \label{test}
\lambda_{n,k} > \sigma^2 \left ( \frac{\beta_{n,p-k}}{n^{2/3}}s(\gamma)+b \right )
\end{eqnarray}
where the value of $s(\gamma)$ can be found by inverting the Tracy-Widom distribution. This distribution has no explicit expression, but can be computed from a solution of a second order Painlev\'e ordinary differential equation. Their estimator is based on a sequence of nested hypothesis tests of the following form: for $k=1,2,\ldots,\mbox{min}(p,n)-1$,
$$\mathcal{H}_0\mbox{: } q_0 \ge k \mbox{ }vs.\mbox{ }\mathcal{H}_1\mbox{: }q_0 \le k-1 \mbox{ .}$$
For each value of $k$, they test the likelihood of the $k$-th eigenvalue $\lambda_{n,k}$ as arising from a signal or from noise as (\ref{test}). If (\ref{test}) is satisfied, $\mathcal{H}_0$ is accepted and $k$ is increased by one. The procedure stops once an instance of  $\mathcal{H}_0$ is rejected and the number of spikes is estimated to be $\widehat{q}_{n,2}=k-1$. Formally, their estimator is defined by
$$\widehat{q}_{n,2}=\mbox{argmin}_k \left ( \lambda_{n,k} < \widehat{\sigma}^2 \left (\frac{\beta_{n,p-k}}{n^{2/3}}s(\gamma)+b \right ) \right ) -1\text{.}$$

When $\sigma^2$ is unknown, they estimate it by the same method we used. For their simulations, they use four different settings, with $\sigma^2=1$ 
\begin{itemlist}
\item A1: $\alpha'=(200,50)$, $c=4$ (i.e. $\alpha=(201,51)$);
\item A2: $\alpha'=(200,50)$, $c=1$;
\item B1: $\alpha'=(200,50,10,5)$, $c=4$ (i.e. $\alpha=(201,51,11,6))$;
\item B2: $\alpha'=(200,50,10,5)$, $c=1$;
\end{itemlist}
with $p=64$ and $p=1024$. Notice that contrary to ours and those of Harding, in their simulation, $c > 1$ and the difference between two consecutive spikes is higher. We add two settings with different variance
\begin{itemlist}
\item A2': $\alpha'=(200,50)$, $c=1$, $\sigma^2=20$ (i.e. $\alpha=(11,3.5)$);
\item B2': $\alpha'=(200,50,10,5)$, $c=1$, $\sigma^2=2$ (i.e. $\alpha=(101,26,6,3.5)$);
\end{itemlist}
and $p=64$. The results are displayed in tables 10 and 11.

\begin{table}[!hb]
\tbl{Summary for $p=64$ showing the frequency of $\hat{q_0}=q_0$.}
{\begin{tabular}{|ccc|}
\hline
$Setting$ & Our estimator & Estimator KN \\
\hline
A1\hphantom{'}; $(p,n)=(64,16)$ & 0.943 & 0.994  \\
A2\hphantom{'}; $(p,n)=(64,64)$ &  0.966 & 0.993  \\
A2'; $(p,n)=(64,64)$ &  0.602 & 0.513   \\
B1\hphantom{'}; $(p,n)=(64,16)$ & 0.348 & 0.238  \\
B2\hphantom{'}; $(p,n)=(64,64)$ & 0.947 & 0.995  \\
B2'; $(p,n)=(64,64)$ & 0.734& 0.682 \\
\hline
\end{tabular}}
\end{table}

With small $p$ and $n$, both estimator performs well, except for the A2', B1, and B2' cases where the spikes are closer to $1+\sqrt{c}$ than in the other cases.

\begin{table}[!ht]
\tbl{Summary for $p=1024$ showing the frequency of $\hat{q_0}=q_0$.}
{\begin{tabular}{|ccc|}
\hline
$Setting$ & Our estimator & Estimator KN\\
\hline
A1; $(p,n)=(1024,256)\hphantom{4}$ & 0.995 & 0.994 \\
A2; $(p,n)=(1024,1024)$ &  0.986 & 0.993\\
B1; $(p,n)=(1024,256)\hphantom{4}$ & 0.999 & 0.999 \\
B2; $(p,n)=(1024,1024)$ & 0.986 & 0.994\\
\hline
\end{tabular}}
\end{table}
~\\
~\\
With larger $p$ and $n$, the results from both methods are comparable. Nevertheless, theoritical properties remain unclear for the KN estimator: it is proved that 
$$\lim_{p,n \rightarrow +\infty} \mathbb{P} \left ( \widehat{q}_{n,2} \ge q_0 \right ) = 1 \text{,}$$
and, in the one factor case ($q_0=1$) that
$$\lim_{p,n \rightarrow +\infty} \mathbb{P} \left ( \widehat{q}_{n,2} > q_0 \right ) = \gamma\text{.}$$
That is by construction, the proposed estimator cannot be fully consistent but nearly consistent with an incompressible asymptotic error of $\gamma$. Actually the authors are using a very small test level $\gamma=0.005$ in their experiments. Whether this property remains true for general case with more than one spike stays open and even so, this near-consistency is a bit unsatisfactory from a theoretical point a view.

\section{Case of Spikes with Multiplicity greater than one}

The problem with two identical spikes is that the difference between the corres\-ponding eigenvalues of the sample covariance matrix will tend to zero. Nevertheless, our method still works: we can explain it by the fact that the convergence of the $\lambda_{n,i}$, for $i > q_0$ (non-spikes) is in $O_{\mathbb{P}} \left ( \frac{1}{n^{2/3}} \right )$, whereas that of the difference corresponding of two identical spikes is in $O_{\mathbb{P}} \left( \frac{1}{\sqrt{n}} \right )$ (Consequence of theorem 3.1 of Bai \& Yao \cite{Bai-Yao}). Furthermore, the variance in the convergence of this difference is $2\alpha^2\left ( 1- \frac{c}{(\alpha-1)^2}\right) \underset{+\infty}{\sim} 2\alpha^2$, which is quite high for high spikes. A complete justification of our method in this case with multiple spikes is still under investigation. Here we provide some simulation results in order to have a first idea about its performance.\\

We will only consider the known variance case. If it is not the case, the procedure explained before will apply without any problem. Here are the results with the same simulation design as previously, except that we introduce multiple spikes. We consider two models:

\begin{itemlist}
\item {\bf Model 3}: $q_0=6$, $(\alpha_1,\alpha_2,\alpha_3,\alpha_4,\alpha_5,\alpha_6)=(259.7,259.7,18,11.1,7.9,4.8)$;
\item {\bf Model 4}: $q_0=6$, $(\alpha_1,\alpha_2,\alpha_3,\alpha_4,\alpha_5,\alpha_6)=(7,6,6,6,5,4)$.
\end{itemlist}

For each model, two different values of $c$, 0.3 and 0.6, are considered, and we give in Figure 5 the frequency of $\hat{q}_n=q_0$ and in Table 12 the mean and the mean squared error of our estimator over 1000 independent replications.\\

\begin{table}[!ht]
\tbl{Mean and mean squared error of $\hat{q}_n$ over 1000 independent replications for Model 1 and 2.}
{\begin{tabular}{|c|cc|cc|}
\cline{2-5}
 \multicolumn{1}{c}{}& \multicolumn{2}{|c|}{Model 3, $q_0=6$} & \multicolumn{2}{c|}{Model 4, $q_0=6$}\\
\hline
$(p,n)$ & Mean & MSE & Mean & MSE\\
\hline
(30,100) & 6.085 & 0.168& 4.529  &  4.393\\
(60,200) & 6.077 & 0.121&  4.86\hphantom{0}  &  4.199\\
(120,400) & 6.088 & 0.082& 5.31\hphantom{0}  &  3.061\\
(240,800) & 6.073 & 0.068& 5.597  &  2.051\\
\hline
(60,100) & 6.043 & 0.151& 4.118 & 4.797\\
(120,200) & 6.092 & 0.108& 4.614 & 4.453\\
(240,400) & 6.081 & 0.074& 5.159 & 3.447\\
(480,800) &6.079 & 0.073& 5.562  &  2.058\\
\hline
\end{tabular}}
\end{table}

\begin{figure}[!ht]
   \includegraphics{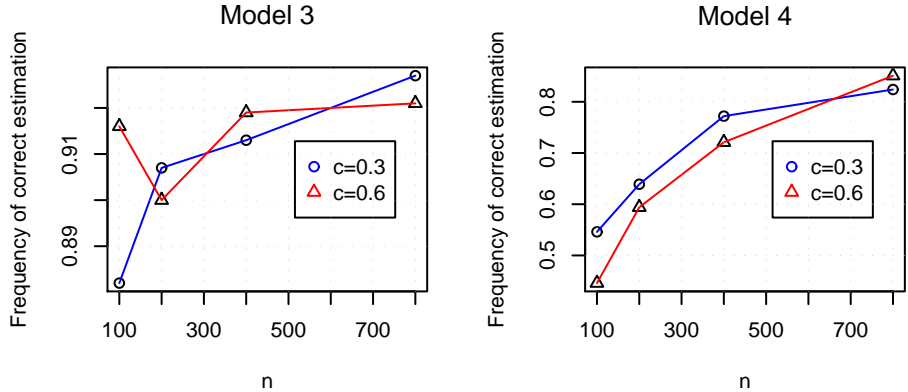}
   \caption{Frequency of $\hat{q}_n=q_0$ over 1000 independent replications.}
\end{figure}

In both cases, we can observe the asymptotic consistency of the estimator, but the convergence is slower in Model 4: indeed, the eigenvalue spacings are smaller. Furthermore, the values of the spikes are small, so that the variance in the convergence of the spikes is not very high and the fluctuations of the difference are smaller than in Model 3.

\section{Extension to the generalized spiked population model}

In \cite{Bai-Yao2}, the author define the generalized spiked population model: the covariance matrix is extended to a general T from I. Once we have 
corresponding Tracy-Widom limits for sample eigenvalues converging to the 
edges of support intervals, our approach can be readily adapted to this 
situation. However such results are lacking.

\end{document}